\documentclass[letterpaper, 10 pt, conference]{ieeeconf}  

\IEEEoverridecommandlockouts                              

\def\urls#1{{\small\url{#1}}}

\def\bfmath#1{{\mathchoice{\mbox{\boldmath$#1$}}%
{\mbox{\boldmath$#1$}}%
{\mbox{\boldmath$\scriptstyle#1$}}%
{\mbox{\boldmath$\scriptscriptstyle#1$}}}}

\def\rd#1{{\color{red}#1}}

\usepackage[colorlinks=true,bookmarks=false,pdfborder={0 0 0},
    linkcolor=black,urlcolor=black,citecolor=black,
    breaklinks=false,bookmarksnumbered=false]{hyperref}
\usepackage{cite} 
\usepackage[utf8]{inputenc}
\usepackage{graphicx}
\graphicspath{ {images/} }
\usepackage{array}
\newcolumntype{P}[1]{>{\centering\arraybackslash}p{#1}}
\usepackage{amsmath}
\usepackage{multirow}
\usepackage{array}
\usepackage{dsfont}
\usepackage{tikz}
\usepackage{amssymb}
\usepackage{amsthm}
\usepackage[ ruled]{algorithm2e}
\usepackage{algcompatible}
\usepackage{float}
\usepackage{multicol}
\usepackage{caption}
\usepackage{subcaption}

\newcommand{\utwi}[1]{\mbox{\boldmath $#1$}}

\renewcommand{\hat}{\widehat}
\renewcommand{\tilde}{\widetilde}
\newcommand{\epsy}{\epsilon}

\newcommand{\diag}{\mathsf{diag}}
\newcommand{\cD}{{\cal D}}

\newcommand{\cL}{{\cal{L}}}

\newcommand{\cX}{{\cal X}}

\newcommand{\bb}{{\bfmath b}}

\newcommand{\bd}{{\bfmath d}}

\newcommand{\bg}{{\bfmath{g}}}
\newcommand{\bh}{{\bfmath{h}}}

\newcommand{\bx}{{\bfmath{x}}}

\newcommand{\bw}{{\bfmath{w}}}

\newcommand{\bz}{{\bfmath{z}}}
\newcommand{\by}{{\bfmath{y}}}
\newcommand{\bA}{{\bfmath A}}

\newcommand{\bC}{{\bfmath{C}}}
\newcommand{\bD}{{\bfmath{D}}}

\newcommand{\bJ}{{\bfmath{J}}}

\newcommand{\bM}{{\bfmath M}}

\newcommand{\bI}{{\bfmath{I}}}

\newcommand{\bW}{{\bfmath W}}
\newcommand{\bU}{{\bfmath U}}

\newcommand{\bV}{{\bfmath V}}

\newcommand{\bgamma}{{\utwi{\gamma}}}

\newcommand{\blambda}{{\utwi{\lambda}}}

\newcommand{\bxi}{{\utwi{\xi}}}

\newcommand{\bGamma}{{\utwi{\Gamma}}}

\newcommand{\reals}{\mathbb{R}}

\newcommand{\sfT}{\textsf{T}}

\newcommand{\proj}{\mathsf{Proj}}

\newtheorem{proposition}{Proposition}
\newtheorem{lemma}{Lemma}
\newtheorem{theorem}{Theorem}

\theoremstyle{definition}
\newtheorem{example}{Example}
\newtheorem{assumption}{Assumption}
\theoremstyle{definition}

\usepackage{epstopdf}

\newcommand{\andrey}[1]{\textcolor{blue}{(Andrey says:  #1)}}
\newcommand{\yue}[1]{\textcolor{violet}{(Yue says:  #1)}}

\def\sfT{{\hbox{\tiny \textsf{T}}}}
\def\eqdef{\mathbin{:=}}

\newcommand{\bfprobe}{\bxi}

\newlength{\noteWidth}
\long\def\notes#1{\ifinner
             {\tiny #1}
             \else
              \marginpar{\parbox[t]{\noteWidth}{\raggedright\tiny #1}}
               \fi}

\begin{document}

\title{\LARGE \bf 
Time-Varying Feedback Optimization for Quadratic Programs with Heterogeneous Gradient Step Sizes}

\author{Andrey Bernstein$^{*}$, Joshua Comden$^{*}$, Yue Chen$^{*}$, Jing Wang$^{*}$%
\thanks{This work was authored by the National Renewable
    Energy Laboratory, managed and operated by Alliance for Sustainable
    Energy, LLC, for the U.S. Department of Energy (DOE) under Contract No.
    DE-AC36-08GO28308. The views
    expressed in the article do not necessarily represent the views of the DOE
    or the U.S. Government. The U.S. Government retains and the publisher, by
    accepting the article for publication, acknowledges that the U.S. Government
    retains a nonexclusive, paid-up, irrevocable, worldwide license to publish
    or reproduce the published form of this work, or allow others to do so, for
    U.S. Government purposes.}
	\thanks{$^{*}$A.~Bernstein, J.~Comden, Y.~Chen, and J.~Wang are with the Power Systems Engineering Center, National Renewable Energy Laboratory, Golden, CO 80401, USA.
		{\tt\small andrey.bernstein@nrel.gov, joshua.comden@nrel.gov,
        yue.chen@nrel.gov,
        jing.wang@nrel.gov}}
}

\maketitle

\begin{abstract}
Online feedback-based optimization has become a promising framework for real-time optimization and control of complex engineering systems. This tutorial paper surveys the recent advances in the field as well as provides novel convergence results for primal-dual online algorithms with heterogeneous step sizes for different elements of the gradient. The analysis is performed for quadratic programs and the approach is illustrated on applications for adaptive step-size and model-free online algorithms, in the context of optimal control of modern power systems.
\end{abstract}


\IEEEpeerreviewmaketitle

\section{Introduction}

This tutorial paper is set in the context of designing real-time optimization and control approaches for efficient operation of complex engineering systems, such as power networks, transportation systems, and communication networks. To illustrate the main concepts, consider an \emph{algebraic} model describing the input-output relation of the system at time $t$:
\begin{align}
\label{eq:model}
\by(t) = \bh_t(\bx(t))
\end{align}
where $\bx(t) \in \mathbb{R}^n$ is a vector of system inputs; $\by(t) \in \mathbb{R}^m$ is the vector of system outputs; and $\bh_t(\cdot): \mathbb{R}^n \rightarrow \mathbb{R}^m$ is the system model\footnote{For example, $\bh_t(\bx(t)) = \bh(\bx(t), \bd(t))$ where $\bd(t)$ is a disturbance process.} map. The efficient operation of the system is then formalized via a time-varying optimization problem\footnote{For simplicity, we assume a stylized optimization where the  performance is measured only with respect to the output $\by$; a general constrained optimization problem is considered in Section \ref{sec:Network}.} 
\begin{align} 
\label{eq:model_opt}
\min_{\bx \in \cX(t), \by = \bh_t(\bx) } f_t(\by)
\end{align}
where $\cX(t)$ is a set representing input constraints (e.g., physical or engineering); and $f_t:  \mathbb{R}^{m}  \rightarrow \mathbb{R}$ is a function quantifying the system performance.

Despite the success of modern optimization theory, directly applying optimization methods to solve \eqref{eq:model_opt} in real time might be impractical for real engineering systems. The major challenges are imprecise knowledge of the system model (including external disturbances) as well as the computational complexity of solving \eqref{eq:model_opt} in real time. 

Online feedback optimization (OFO) \cite{opfPursuit,Hauswirth2016,Hauswirth2018,bernstein2019,colombino2019,chen2020,haberle2020,he2022modelfree} is an emerging framework that aims at tackling the above challenges. The essential idea of OFO boils down to implementing the optimization algorithms in a feedback loop with the actual system. To illustrate the idea, consider first a standard (feedforward) optimization approach to seek solutions (or stationary points) of  \eqref{eq:model_opt} via a projected-gradient method that at each iteration $k = 0, 1, ..$ performs
\begin{align}
\bx^{(k+1)} = & \proj_{\cX(t)}\Big\{ \bx^{(k)} - \alpha (\bJ_t^{(k)})^\sfT \nabla_{\by} f_t (h_t(\bx^{(k)}))  \Big\}, \label{eq:offline}
\end{align}
where $\bJ_t^{(k)} := \bJ_{\bh_t} (\bx^{(k)})$ is the Jacobian matrix of  $\bh_t$; $\proj_{\cX}(\bz) := \arg \min_{\bx \in \cX} \|\bz - \bx\|_2$ is the projection operator; and $\alpha > 0$ is the step size. In the OFO framework, \eqref{eq:offline} is replaced with an \emph{online} update rule
\begin{align}
\bx^{(k+1)} = & \proj_{\cX^{(k)}}\Big\{ \bx^{(k)} - \alpha (\bJ^{(k)})^\sfT \nabla_{\by} f^{(k)}(\hat{\by}^{(k)})  \Big\}, \label{eq:online}
\end{align}
performed at each discrete time instance $t_k$ (the iteration index  is now identified with a time step). In \eqref{eq:online},  $\cX^{(k)} := \cX(t_k)$, $\bJ^{(k)} := \bJ_{\bh_{t_k}} (\bx^{(k)}) $, $f^{(k)} := f_{t_k}$; and $\hat{\by}^{(k)}$ is the measurement of the system output $\bh_{t_k}(\bx^{(k)})$ representing the feedback. Note that \eqref{eq:online} eliminates the explicit dependence on the system model by the virtue of feedback but still requires the real-time knowledge of the system Jacobian  $\bJ^{(k)}$. 

The existing OFO methods can be classified according to their application scenarios and implementation details; see, e.g., \cite{he2022modelfree,zhan2023} for a survey. In particular, the following aspects can be used to classify the OFO methods:
\begin{enumerate}
    \item \textbf{Model type: algebraic or dynamic.}  This paper focuses on the algebraic model \eqref{eq:model}. For OFO methods with dynamic system models, e.g., see \cite{he2022modelfree,colombino2019}.
    \item \textbf{Model information.} To implement \eqref{eq:online}, the gradient of $F^{(k)} (\bx) := f^{(k)} (\bh_{t_k} (\bx))$, or approximation thereof, needs to be computed. There are two main approaches in the literature:
    \begin{enumerate}
        \item \textbf{Jacobian estimation.} The Jacobian matrix $\bJ^{(k)}$ is estimated either based on a partial/approximate system model (e.g., linearization of \eqref{eq:model} around a given operating point) \cite{bernstein2019} or leveraging data-driven/machine-learning-based approaches \cite{nonhoff2021}. Denoting by $\hat{\bJ}$ the approximate Jacobian, the following proxy is then used in \eqref{eq:online} for the gradient of the objective function:
        \begin{equation}
            \hat{\nabla} F^{(k)} :=\hat{\bJ}^\sfT \nabla_{\by} f^{(k)}(\hat{\by}^{(k)}) 
            \label{eq:grad_proxy1}
        \end{equation}
        
        \item \textbf{Zero-order approximation.} The gradient of $F^{(k)}$ is directly approximated using one of the many zero-order approximation approaches. For example, in \cite{chen2020}, a two-point  approximation was used:
\begin{equation} \hspace{-1.5cm}
    \hat{\nabla} F^{(k)} \eqdef   \frac{1}{2\epsy} \bxi \left[ F^{(k)} (\bx^{(k)} + \epsy \bxi) - F^{(k)}(\bx^{(k)} - \epsy \bxi) \right]
    \label{eq:grad_proxy2}
\end{equation}
where $\bfprobe \in \reals^n$ is a perturbation (or exploration) vector,  and $\epsy > 0$ is a (small) scalar. Other approaches include the one-point  estimation \cite{bernstein2019a} and the residual one-point estimation \cite{he2022modelfree}. In any case, the obtained proxy for the gradient only requires measurements of the system output feedback. 
    \end{enumerate}

    \item \textbf{Algorithm/problem type.} Here the distinction is usually between primal and primal-dual (or dual) methods. Primal methods such as \eqref{eq:online} are effective in solving optimization problems where no output constraints are present or when the projection on the constraints can be computed efficiently and centrally \cite{he2022modelfree}. Primal-dual methods can deal efficiently with complicated  output constraints with the caveat of only asymptotic constraint satisfaction. Also, primal-dual methods better lend themselves to distributed implementation \cite{bernstein2019}. 
    \item \textbf{Problem time variability.} In static settings, where the objective and constraints in \eqref{eq:model_opt} do not vary with time, the proposed algorithms typically are shown to converge to the optimal solution/stationary point, or neighborhood thereof. In dynamic settings, the algorithms are typically designed to track the time-varying optimal solution/stationary point of \eqref{eq:model_opt} up to a well-defined error.
    \item \textbf{Algorithm step size.} Most of the literature, with the exception of \cite{comden2023adaptive,haberle2020}, focuses on algorithms with a uniform (single) step size parameter; cf. $\alpha$ in \eqref{eq:online}. This assumption  simplifies significantly the convergence analysis, especially in the constrained time-varying case. In practice, however, it is beneficial to have different step sizes for different elements of $\hat{\nabla} F^{(k)}$. Moreover, adaptive and heterogeneous step sizes have been shown to improve numerical performance of the OFO algorithms on specific test cases \cite{comden2023adaptive}.
\end{enumerate}

\emph{Contributions.} In addition to surveying the recent advances of OFO methods, this paper aims at addressing the last challenge above, namely providing convergence guarantees for time-varying OFO methods with heterogeneous step sizes. Specifically, we consider the online OFO algorithm of the form
\begin{align}
\bx^{(k+1)} = & \proj_{\cX^{(k)}}\Big\{ \bx^{(k)} - \alpha \bGamma^{(k)} \hat{\nabla} F^{(k)}  \Big\}, \label{eq:online_hetero}
\end{align}
where $\bGamma^{(k)} = \diag (\bgamma^{(k)})$ is a diagonal positive definite matrix. Most of our analysis focuses on quadratic programs (QPs) and linear algebraic system models due to reasons that will become apparent next. 

While the prototypical problem \eqref{eq:model_opt} and algorithm \eqref{eq:online_hetero} are used here to present the main concepts, the paper considers constrained problems and  online primal-dual solution methods. Finally, to illustrate the performance of the developed methods, we apply them to the
problem of optimal operation of power distribution network with distributed energy resources.

\paragraph*{Organization}
The rest of the paper is organized as follows. In Section \ref{sec:Network}, we describe the essentials of the OFO framework. In Section \ref{sec:hetero}, we introduce the algorithms with heterogeneous step sizes and provide conditions for convergence. In Section \ref{sec:power_systems}, we showcase two applications of heterogeneous step-size OFO in the context of power systems optimization. Finally, we conclude in Section \ref{sec:conclusion}.

\section{Time-Varying OFO Framework}
\label{sec:Network}

\subsection{Target Optimization Problem}
Consider a network of $N$ systems, whose  objectives are defined via the following optimization problem:
\begin{subequations} 
\label{eqn:sampledProblem}
\begin{align} 
 &\min_{\substack{\bx \in \reals^n, \by \in \reals^m}}\hspace{.2cm} f_0^{(k)}(\by) + \sum_{i = 1}^N f_i^{(k)}(\bx_i) \label{eq:obj_p0} \\
&\mathrm{subject\,to:~} \bx_i \in \cX_i^{(k)} , \, i = 1, \ldots, N \label{eqn:constr_Xsampl} \\
& \hspace{1.9cm} \by = \bh^{(k)} (\bx) \label{eqn:constr_sys} \\
& \hspace{1.9cm} g_j^{(k)}(\by)  \leq 0 , \, j = 1, \ldots, M \label{eq:ineqconst}
 \end{align}
\end{subequations}
where $\cX_i^{(k)} \subset \mathbb{R}^{n_i}$, with $\sum_{i = 1}^N n_i = n$, are convex sets defining input  constraints for system $i$; and $\by = \bh^{(k)} (\bx)$ is the algebraic model for the networked system output. This paper mostly focuses on the linear (or linearized) case
\begin{equation} \label{eq:sys}
\bh^{(k)}(\bx) := \bC \bx + \bU \bw^{(k)} \in \reals^m    
\end{equation}
 where $\bC \in \mathbb{R}^{m \times n}$ and $\bU \in \mathbb{R}^{m \times w}$ are given model parameters, and $\bw^{(k)} \in \mathbb{R}^{w}$ is a vector of time-varying exogenous (uncontrollable) inputs. 
For example, in the power systems case considered in Section \ref{sec:power_systems}, $\bh^{(k)}(\bx)$ is the linear power-flow model.

In \eqref{eqn:sampledProblem}, $f_0^{(k)}(\bh^{(k)}(\bx))$ is a convex function (in terms of $\bx$) defining the objective in terms of the outputs; and 
 $f_i^{(k)}(\bx_i)$ are convex functions defining the objectives of subsystems in terms of their input. Finally,  $g_j^{(k)}(\bh^{(k)}(\bx))$ are  convex (in terms of $\bx$) output-constraints functions. 

To simplify notation, we let:
\[
\begin{aligned}
\bg^{(k)}(\bh^{(k)}(\bx))  & := \left[g_1^{(k)}(\bh^{(k)}(\bx)), \ldots, g_M^{(k)}(\bh^{(k)}(\bx))\right]^\sfT
\\[.5em]
\phi^{(k)}(\bx) & := \sum_i f_i^{(k)}(\bx_i) 
\\
f^{(k)} (\bx) & := \phi^{(k)}(\bx) + f_0^{(k)}(\bh^{(k)} (\bx)). 
\end{aligned} 
\]
Let $\blambda \in \mathbb{R}_+^M$ denote the vector of dual variables associated with~\eqref{eq:ineqconst}, and 
\begin{align}
\cL^{(k)}(\bx, \blambda) & :=f^{(k)} (\bx) + \blambda^\sfT\bg^{(k)}(\bh^{(k)}(\bx)) \, 
\label{eqn:lagrangian}
\end{align}
denote the associated Lagrangian function. 
The \emph{regularized} Lagrangian is given by \cite{Koshal11}:
\begin{align} \label{eq:lagrangian_reg}
\cL^{(k)}_{p,d}(\bx, \blambda) := \cL^{(k)}(\bx, \blambda) + \frac{p}{2}\|\bx\|_2^2 - \frac{d}{2}\|\blambda\|_2^2 
\end{align}
for some $p > 0$ and $d > 0$. 

Consider then the saddle-point problem:
\begin{align} \label{eq:minmax}
\max_{{\footnotesize \blambda} \in \cD^{(k)}} \min_{\bx \in \cX^{(k)}} \cL^{(k)}_{p,d}(\bx, \blambda) \, \hspace{.5cm} k \in \mathbb{N}
\end{align}     
where $\cX^{(k)} := \cX_1^{(k)} \times \ldots \times \cX_N^{(k)}$ and $\cD^{(k)}$ is a convex and compact set that contains the optimal dual variable of \eqref{eq:minmax}  \cite{bernstein2019}. We interpret the optimal trajectory $\bz^{(*, k)} := \{\bx^{(*,k)}, \blambda^{(*,k)}\}_{k\in \mathbb{N}}$ of \eqref{eq:minmax} as the \emph{reference} trajectory which the online algorithms are designed to track. 
Since  $\cL^{(k)}_{p,d}(\bx, \blambda)$ is strongly convex in $\bx$ and strongly concave in $\blambda$ by design, the  optimizer $\bz^{(*, k)}$ of~\eqref{eq:minmax}  is unique. 
However, it typically would be different from any saddle point of the Lagrangian $\cL^{(k)}(\bx, \blambda)$ (and therefore from any solution of \eqref{eqn:sampledProblem}); the difference can be bounded as, e.g., in \cite{Koshal11}, as a function of the regularization parameters $p$ and $d$.

The following assumptions are typically imposed\footnote{We list these assumptions here for completeness; not all the assumptions are necessary for specific convergence results.} on the different elements of \eqref{eqn:sampledProblem} to establish convergence/tracking guarantees.

\begin{assumption} 
\label{ass:constqualification}
Slater's condition holds at each time step $k$. 
\end{assumption}

\begin{assumption} 
\label{ass:cost}
The functions $f_0^{(k)}(\by)$ and $f^{(k)}_i(\bx_i)$ are convex and continuously differentiable. The  maps $\nabla \phi^{(k)}(\bx)$, $\nabla f_0^{(k)}(\by)$, and $\nabla^2 f_0^{(k)}(\by)$  are Lipschitz continuous.
\end{assumption}
\begin{assumption} 
\label{ass:nonlinearconstr}
For each $j = 1, \ldots, M$, the function $g_j^{(k)}(\by)$ is convex and continuously differentiable.  Moreover, $\nabla g_j^{(k)}(\by)$ and $\nabla^2 g_j^{(k)}(\by)$  are Lipschitz continuous. 
\end{assumption}

\begin{assumption} \label{ass:grad_var}
There exists a constant $e_f$ such that for all $k$ and $\bx, \by$:
\begin{align}
    \|\nabla f_0^{(k)}(\by) - \nabla f_0^{(k-1)}(\by)\| &\leq e_f, \nonumber \\
    \|\nabla \phi^{(k)}(\bx) - \nabla \phi^{(k-1)}(\bx)\|  &\leq e_f,   \nonumber \\
    \| \nabla g_j^{(k)}(\by) - \nabla g_j^{(k-1)}(\by)\| &\leq e_f, \quad j = 1, \ldots, M. \nonumber
\end{align}
\end{assumption}

\subsection{Feedback Primal-Dual Methods}
\label{sec:pri-dual}

To track the reference trajectory $\{\bz^{(*, k)}\}$, a feedback-based primal-dual method is typically used in the OFO framework \cite{opfPursuit,bernstein2019}:
\begin{subequations} 
\label{eq:primalDualFeedback}
\begin{align}
\bx^{(k+1)} & = \proj_{\cX^{(k)}}\Big\{\bx^{(k)} - \alpha \hat{\nabla}_\bx\cL_{p,d}^{(k)} \Big\}  \label{eq:primalstep} \\
\blambda^{(k+1)} & = \proj_{\cD^{(k)}}\Big\{ \blambda^{(k)} + \alpha \hat{\nabla}_\lambda\cL_{p,d}^{(k)} \Big\} \label{eq:dualstep}
\end{align}
\end{subequations}
where $\alpha > 0$ is a constant step size, and  $\hat{\nabla}_\bx\cL_{p,d}^{(k)}$ and $\hat{\nabla}_\lambda\cL_{p,d}^{(k)}$ are feedback-based proxies for the  gradients of the regularized Lagrangian (cf.~\eqref{eq:grad_proxy1} and \eqref{eq:grad_proxy2}). In particular, if the system Jacobian matrix $\bC$ (or estimation thereof) is used, we have that
\begin{subequations} 
\label{eq:feedbackGradients}
\begin{align}
\hat{\nabla}_\bx\cL_{p,d}^{(k)} & =   \nabla_\bx \phi^{(k)}(\bx^{(k)}) + \bC^\sfT\nabla_\by f_0^{(k)}(\hat{\by}^{(k)} )  \nonumber \\  
& \hspace{.4cm}
 + \sum_{j = 1}^M \lambda_j^{(k)}\bC^\sfT\nabla g_j^{(k)}(\hat{\by}^{(k)} ) + p \bx^{(k)} \label{eq:feedbackGradientsPrimal}\\
\hat{\nabla}_\lambda\cL_{p,d}^{(k)} & = \bg^{(k)}(\hat{\by}^{(k)} ) - d \blambda^{(k)} \label{eq:feedbackGradientsDual}
\end{align}
\end{subequations}
where $\hat{\by}^{(k)}$ is a  measurement of $\bh^{(k)}(\bx^{(k)})$. Alternatively, a zero-order  proxy can be used for $\hat{\nabla}_\bx\cL_{p,d}^{(k)}$ instead of \eqref{eq:feedbackGradientsPrimal}. For example, using a two-point gradient estimation \eqref{eq:grad_proxy2}, we have \cite{chen2020}:
\begin{align}
    \hat{\nabla}_\bx\cL_{p,d}^{(k)} &:= \nabla_\bx \phi^{(k)}(\bx^{(k)}) + p \bx^{(k)}\nonumber\\
    &\quad + \frac{1}{2 \epsy}\bfprobe^{(k)}\left[f_0^{(k)}(\hat{\by}^{(k)}_+ ) -f_0^{(k)}(\hat{\by}^{(k)}_- )  \right] \label{eq:zeroOrderGradients}\\
    &\quad + \frac{1}{2 \epsy} \bfprobe^{(k)}(\blambda^{(k)})^\sfT \left[\bg^{(k)}(\hat{\by}^{(k)}_+ ) - \bg^{(k)}(\hat{\by}^{(k)}_- )  \right], \nonumber
\end{align}
 where $\{\bfprobe^{(k)}\}$ is an exploration process (either deterministic \cite{bernstein2019a} or stochastic \cite{spa92}); 
$\epsy > 0$ is a (small) scalar controlling the magnitude of exploration; and  $\hat{\by}^{(k)}_\pm$ are the measurements of the system output at 
 $\bx^{(k)}_\pm := \bx^{(k)} \pm \epsy \bfprobe^{(k)}$. 

 In this paper, we are focusing on the \emph{distributable} case that is quantified by the following assumption.

 \begin{assumption} \label{as:distr}
     The target optimization problem is distributable in the sense that the associated Lagrangian function, and regularizations and approximations thereof, are decomposable across different subsystems, i.e.:
     \begin{align*}
 \nabla_{\bx_i} \cL^{(k)}_{p,d}(\bx, \blambda) &= \nabla_{\bx_i} \cL^{(k)}_{i, p,d}(\bx_i, \blambda)
     \end{align*}
     for all $i$. 
 \end{assumption}
Under Assumption \ref{as:distr}, the primal step of the OFO algorithm \eqref{eq:primalstep} can be written as
\begin{align} \label{eq:algo_distr}
\bx^{(k+1)}_i & = \proj_{\cX_i^{(k)}}\Big\{\bx_i^{(k)} - \alpha \hat{\nabla}_{\bx_i}\cL_{p,d}^{(k)} \Big\}, \quad i = 1, \ldots, N. 
\end{align}

Note that \eqref{eq:algo_distr} and \eqref{eq:dualstep} can be implemented in a distributed, gather-and-broadcast approach, wherein \eqref{eq:dualstep} is computed centrally and \eqref{eq:algo_distr} is computed locally at each subsystem; see, e.g., \cite{unified} for details.

\subsection{Key Elements of Convergence Analysis for Time-Varying OFO} \label{sec:key-elems}

In this section, we present typical results for tracking the reference trajectory $\{\bz^{(*, k)}\}$ and discuss the key elements in the their proofs. 
Let $e_y$ denote the worst-case error in the output feedback measurement. Then the following result holds \cite{bernstein2019,chen2020}.

\begin{theorem} \label{thm:conv} Let $\bz^{(k)} := \{\bx^{(k)}, \blambda^{(k)}\}_{k\in \mathbb{N}}$.
Under Assumptions 1-4, there exists $\bar{\alpha} > 0$ such that if the step size $\alpha$ satisfies $0 < \alpha  < \bar{\alpha}$, the sequence  $\{\bz^{(k)}\}$ converges Q-linearly to $\{\bz^{(*,k)}\}$ up to an asymptotic error bound given by:
\begin{align} 
\limsup_{k\to\infty} \|\bz^{(k)} - \bz^{(*, k)}\|_2 &\leq \frac{\alpha  \epsy_\phi + \sigma}{1 - c}  \label{eqn:asym_bound_apr}
\end{align}
where
$c < 1$,
$\sigma \eqdef \max_k \|\bz^{(*, k+1)} - \bz^{(*, k)}\|_2$, and $\epsy_\phi = O(e_y)$ when the estimated gradient \eqref{eq:feedbackGradients} is used; or
    $\epsy_\phi =  O(\alpha + \epsy^2 + e_f + e_y)$ when the estimated gradient \eqref{eq:zeroOrderGradients} is used.
\end{theorem}

To prove such a theorem, standard time-varying fixed-point iteration analysis is typically used; see, e.g., \cite{fixedPoints}. To this end, the following two properties of the primal-dual operator
\begin{align}
\Phi^{(k)}(\bz) &\eqdef \left[\begin{array}{c}
 \nabla_{\bx}  \cL^{(k)}_{p, d}(\bz) \\
 - \nabla_{\tiny \blambda}  \cL^{(k)}_{p, d}(\bz)
\end{array}
\right], \quad \bz := (\bx, \blambda)  
\end{align}
are  established.

\subsubsection{Fixed-Point of the Algorithm is the Saddle Point}

Consider the fixed-point equation associated with the standard primal-dual algorithm applied to $\cL^{(k)}_{p, d}$:
\begin{align}
    \bz = & \proj_{\cX^{(k)} \times \cD^{(k)}}\Big\{ \bz - \alpha \Phi^{(k)}(\bz)  \Big\}.
    \label{eq:fixedPoint}
\end{align}
It is well-known \cite{opfPursuit, bernstein2019} that under the assumptions in this paper, the unique fixed point of \eqref{eq:fixedPoint} coincides with the unique saddle point of \eqref{eq:minmax}. 

\subsubsection{Strong Monotonicity of $\Phi^{(k)}$}
Under the assumptions in this paper, the operator $\Phi^{(k)}$ is strongly monotone \cite{opfPursuit, bernstein2019}. Namely, there exists $\eta > 0$ such that
\begin{equation}
    \left(\Phi^{(k)} (\bx_1) - \Phi^{(k)} (\bx_2) \right)^\sfT (\bx_1 - \bx_2) \geq \eta \|\bx_1 - \bx_2 \|_2^2
\end{equation}
for all $\bx_1, \bx_2$ and all $k$.

\section{Algorithms with Heterogeneous Step Sizes} \label{sec:hetero}

In this section, we illustrate the fact that even in the simple case of (primal) projected gradient method, heterogeneous step sizes might lead to convergence to a suboptimal point. We then discuss the conditions under which convergence to the optimum can be guaranteed and extend the results to primal-dual methods. Most of the analysis in this section is restricted to target optimization problems \eqref{eqn:sampledProblem} which can be written in the QP form\footnote{In this section, we omit the time variability of the optimization problem for brevity and focus on static problems; however, the general framework applies to time-varying problems as mentioned in Section \ref{sec:conv-ofo-hetero}.}:
\begin{subequations} 
\label{eq:QP}
\begin{align} 
 &\min_{\substack{\bx \in \cX}}\hspace{.2cm} f(\bx) \label{eq:QP_obj} \\
&\mathrm{subject\,to:~} \bg(\bx) \leq 0 \label{eq:QP_constr} 
 \end{align}
\end{subequations}
where
\begin{equation}
    f(\bx) = \frac{1}{2}\bx^\sfT \bA \bx + \bb^\sfT \bx + c, \quad \bA \succcurlyeq 0
    \label{eq:QP_objfun}
\end{equation}
and
\begin{equation}
    \bg(\bx) = \bD \bx + \bd.
\end{equation}
We note that QPs capture a large set of practical problems that are of interest in our application domains. For example, the linearized optimal power flow problem discussed in Section \ref{sec:power_systems} is a QP.
Extension of the results to more general problems is left for future research.

\subsection{Projected Gradient Method}

We first consider the case where the constraint \eqref{eq:QP_constr} is absent and analyze the convergence of the projected gradient algorithm
\begin{align}
\bx^{(k+1)} = & \proj_{\cX}\Big\{ \bx^{(k)} - \alpha \bGamma^{(k)}  \nabla f (\bx^{(k)})  \Big\},
\end{align}
where $\bGamma^{(k)} $ is a diagonal positive definite matrix of step sizes and $\alpha > 0$ is the step-size scaling factor. We next discuss the two key properties for convergence analysis presented in Section \ref{sec:key-elems}.

\subsubsection{Optimality of Fixed Points}
Note that without the projection, this algorithm converges to a solution of \eqref{eq:QP} under standard conditions on the scaling factor $\alpha$ (e.g., using the analysis in \cite{fixedPoints}). However, this is not the case with the projection in general. Specifically, the set of fixed points of 
\begin{align}
\bx = & \proj_{\cX}\Big\{ \bx - \alpha \bGamma^{(k)} \nabla f (\bx)  \Big\},
\end{align}
does not necessarily coincide with the set of optimal points, as the following example shows.


\begin{example}
    Suppose that $\cX:=\Big\{\bx\in\mathbb{R}^2:\mathbf{1}^\sfT\bx\geq 8\Big\}$, $\bA:=\bI$, $\bb:=\mathbf{0}$, and $c:=0$.
    If $\bGamma^{(k)} (\bx):=\bI$, then the fixed point is $[4~~4]^\sfT$, and is on the boundary of $\cX$ with the gradient of the objective function being $[4~~4]^\sfT$.
    It also happens to be optimal with the objective value of 16.
    However, if instead $\bGamma^{(k)}$ is slightly perturbed away from $\bI$, i.e., $\bGamma^{(k)} (\bx):=\diag([\frac{3}{4}~~\frac{5}{4}]^\sfT)$, then the fixed point is $[5~~3]^\sfT$, and is on the boundary of $\cX$ with the gradient of the objective function being $[5~~3]^\sfT$.  Also, it is not optimal with the objective value of 17. \qed
\end{example}

To tackle this, we slightly redefine the projected gradient algorithm as follows:
\begin{align}
\bx^{(k+1)} = & \proj_{\cX}\Big\{ \bx^{(k)} - \alpha \bGamma^{(k)} (\bx^{(k)})  \nabla f (\bx^{(k)})  \Big\},
\end{align}
where $\bGamma^{(k)} (\bx)$ is now defined via:
\begin{align}
    \bGamma^{(k)} (\bx) := \begin{cases}
 	    \diag(\bgamma^{(k)}) & \text{if  $\bx - \alpha \diag(\bgamma^{(k)}) \nabla f (\bx) \in \cX,  $ }\\
 	    \bI & \text{otherwise}. \label{eq:Gamma_k}
       \end{cases}  
\end{align}
We then have the following.
\begin{proposition} \label{prop:fixed_points}
    For each $k$, the set of fixed points of 
  \begin{align} \label{eq:fixed_points_Gamma}
\bx = & \proj_{\cX}\Big\{ \bx - \alpha \bGamma^{(k)} (\bx) \nabla f (\bx)  \Big\}
\end{align}  
coincides with the set of fixed points of 
  \begin{align} \label{eq:fixed_points}
\bx = & \proj_{\cX}\Big\{ \bx - \alpha  \nabla f (\bx)  \Big\}.
\end{align}  
Therefore, the set of fixed points of \eqref{eq:fixed_points_Gamma} coincides with the set of the optimal solutions to $\min_{\bx \in \cX} f(\bx)$.
\end{proposition}
The proof  leverages the following Lemma that follows from the standard arguments of optimiality of convex problems.
\begin{lemma} \label{lem:opt}
If $\bx^* \in \arg \min_{\bx \in \cX} f(\bx)$ with $\nabla f(\bx^*) \neq \mathbf{0}$, then:
\begin{align}
    \bx^* - \alpha \bGamma \nabla f(\bx^*) \notin \cX
\end{align}
for any diagonal positive definite matrix $\bGamma$ and any $\alpha > 0$.
\end{lemma}

\begin{proof}[Proof of Proposition \ref{prop:fixed_points}]
    Let $\bx^*$ satisfy \eqref{eq:fixed_points_Gamma}. If $\bx^* - \alpha \diag(\bgamma_k) \nabla f (\bx^*) \in \cX$ then $\bx^* = \bx^* - \alpha \diag(\bgamma_k) \nabla f (\bx^*) $ and $\nabla f(\bx^*) = \mathbf{0}$. If $\bx^* - \alpha \diag(\bgamma_k) \nabla f (\bx^*) \notin \cX$, then $\bx^* =  \proj_{\cX}\Big\{ \bx^* - \alpha  \nabla f (\bx^*)  \Big\}$. Therefore, $\bx^*$ satisfies \eqref{eq:fixed_points}.
   
    Now let $\bx^*$ satisfy \eqref{eq:fixed_points}. If $\nabla f(\bx^*) = \mathbf{0}$, $\bx^*$ clearly satisfies \eqref{eq:fixed_points_Gamma}. If $\nabla f(\bx^*) \neq \mathbf{0}$, then using Lemma \ref{lem:opt}, we have that $\bx^* - \alpha  \diag(\bgamma^{(k)})  \nabla f(\bx^*) \notin \cX$, and $\bx^*$  satisfies \eqref{eq:fixed_points_Gamma} with $\bGamma^{(k)}(\bx^*) = \bI$.
\end{proof}


\subsubsection{Strong Monotonicity}

Consider the operator
\begin{equation}
    \Phi(\bx) \eqdef \bGamma  \nabla f (\bx). 
\end{equation}
The following example shows that $\Phi(\bx)$ is not necessarily monotone even in the case of quadratic and strongly convex objective \eqref{eq:QP_objfun}.

\begin{example}
    Suppose that $f(\bx) = \frac{1}{2}\bx^\sfT \bA \bx$ for some positive definite and symmetric matrix $\bA$. Clearly, $f$ is strongly convex and we have
    \[
    \Phi(\bx) = \bGamma \bA \bx.
    \]
    To check if $\Phi(\bx)$ is strongly monotone, we need to establish at least the monotonicity of $\Phi(\bx)$ in the sense that for any $\bx_1,\bx_2$, 
    \[
    (\Phi(\bx_1) - \Phi(\bx_2))^\sfT (\bx_1 - \bx_2) \geq 0
    \]
    which is equivalent to
    \[
    (\bx_1 - \bx_2)^\sfT (\bA \bGamma)(\bx_1 - \bx_2) \geq 0.
    \]
    The latter holds true if and only if $0.5(\bGamma \bA + \bA\bGamma  )$ is positive semidefinite. However, this is not the case in general. For example, consider
    \[
    \bA = \begin{bmatrix}
        2 & -1 \\ -1 & 2
    \end{bmatrix}, \,
    \bGamma = \begin{bmatrix}
        \delta & 0 \\ 0 & 1
    \end{bmatrix}.
    \]
    If $\delta = 20$, we have that
    \[
    \frac{1}{2} \left(\bGamma \bA + \bA\bGamma \right) = \begin{bmatrix}
        40 & -10.5 \\ -10.5 & 2
    \end{bmatrix}
    \]
    with eigenvalues $42.7$ and $-0.7$. In this case, the more unbalanced $\bGamma$ is (namely, the larger $\delta$ is), the smaller is the second eigenvalue. Specifically, for $\delta \geq 14$, the matrix $0.5(\bGamma \bA + \bA\bGamma  )$ becomes indefinite.
  \qed
\end{example}

To enforce strong monotonicity of $\Phi(x)$, we use the regularization approach similar to \eqref{eq:lagrangian_reg}. In particular, we consider the following regularized objective function:
\begin{equation}
    \tilde{f}(\bx) := f(\bx) + \frac{p}{2} \bx^\sfT \bGamma^{-1} \bx 
\label{eq:reg_obj}
\end{equation}
for some $p > 0$. Consider the operator
\begin{align}
    \tilde{\Phi}(\bx) &\eqdef \bGamma  \nabla \tilde{f} (\bx) \\
    &= (\bGamma \bA +p \bI) \bx + \bGamma \bb.
\end{align}
We then have the following result to esnure strong monotonicty of $\tilde{\Phi}$.
\begin{proposition} \label{prop:sm-gd}
Let
\begin{equation}
    \bM \eqdef \frac{1}{2} \Big (\bGamma \bA + \bA \bGamma  \Big)
\end{equation}
and let $\lambda_{\min}$ be the smallest eigenvalue of $\bM$. If $p > \max \{0, -\lambda_{\min} \}$, then $\tilde{\Phi}$ is strongly monotone with parameter $\eta := p + \min \{0, \lambda_{\min} \}$.
\end{proposition}

\begin{proof}
    We have that
\begin{align}
&\left (\tilde{\Phi}(\bx_1) - \tilde{\Phi}(\bx_2) \right)^\sfT (\bx_1 - \bx_2) \notag \\
&= \left [ \left ( \bGamma \bA + p \bI \right ) (\bx_1 - \bx_2) \right]^\sfT (\bx_1 - \bx_2) \notag \\
&= (\bx_1 - \bx_2)^\sfT \left(  \bA \bGamma + p \bI \right )(\bx_1 - \bx_2) \notag\\
&= (\bx_1 - \bx_2)^\sfT \left(  \bM + p \bI \right )(\bx_1 - \bx_2). \label{eq:mono_pg}
\end{align}
If $p > \max \{0, -\lambda_{\min} \}$, the matrix $\bM + p \bI$ is positive definite, and specifically
\[
\bM + p \bI - \eta \bI \succcurlyeq 0
\]
for $\eta$ define above. Hence
\begin{align*}
 (\bx_1 - \bx_2)^\sfT \left(  \bM + p \bI - \eta \bI\right )(\bx_1 - \bx_2) \geq 0   
\end{align*}
implying that 
\begin{align}
 (\bx_1 - \bx_2)^\sfT \left(  \bM + p \bI \right )(\bx_1 - \bx_2) \geq \eta \| \bx_1 - \bx_2\|^2. \label{eq:mono_pg1}
\end{align}
Plugging \eqref{eq:mono_pg1} in \eqref{eq:mono_pg} completes the proof.
\end{proof}

\subsection{Primal-Dual Method}
Next we extend the results from the previous section to a primal-dual algorithm for solving \eqref{eq:QP}. 
To this end, let 
\begin{equation}
    \cL(\bx, \blambda) = f(\bx) + \blambda^\sfT \bg (\bx) 
\end{equation}
denote the Lagrangian associated with \eqref{eq:QP} and assume that the problem is distributable (Assumption \ref{as:distr}).  Similarly to \eqref{eq:reg_obj}, we consider the regularized Lagrangian
\begin{align} \label{eq:regLagHet}
    \cL_{p}(\bx, \blambda) \eqdef \cL(\bx, \blambda) + \frac{p}{2} \bx^\sfT \bGamma_x^{-1} \bx - \frac{p}{2} \blambda^\sfT \bGamma_{\lambda}^{-1} \blambda.
\end{align}
and the corresponding saddle-point problem. Note that, compared to \eqref{eq:lagrangian_reg}, we use the same parameter $p$ for both primal and dual part for simplicity. We next provide conditions under which the saddle point of $\cL_{p}(\bx, \blambda)$  coincides with the solution of the corresponding fixed-point equation as well as conditions for strong monotonicity of the primal-dual operator associated with $\cL_{p}(\bx, \blambda)$.

\subsubsection{Optimality of Fixed Points}
Consider the projected primal-dual algorithm
\begin{subequations} \label{eq:mod_pd}
\begin{align} 
\bx_i^{(k+1)} = & \proj_{\cX_i}\Big\{ \bx_i^{(k)} - \alpha_i \bGamma_{x, i}  ^{(k)} (\bx_i^{(k)}) \nabla_{\bx_i} \cL_{i, p} ( \bx_i^{(k)}, \blambda^{(k)})  \Big\},\\
\blambda^{(k+1)} = & \proj_{\cD}\Big\{ \blambda^{(k)} + \alpha \bGamma_{\lambda}^{(k)} ( \blambda^{(k)} )  \nabla_{\tiny \blambda} \cL_p( \bx^{(k)}, \blambda^{(k)})  \Big\},
\end{align}
\end{subequations}
where $\bGamma_{x, i}  ^{(k)}$ and $\bGamma_{\lambda}^{(k)}$
are defined similalry to \eqref{eq:Gamma_k}.

\begin{proposition}
     The unique fixed point associated with \eqref{eq:mod_pd} coincides with the unique saddle point of $\cL_{p}$.
\end{proposition}
The proof is similar to that of Proposition \ref{prop:fixed_points}.

\subsubsection{Strong Monotonicity}
Let $\bz := (\bx^\sfT, \blambda^\sfT)^\sfT$ and
\begin{align}
    \tilde{\Phi}(\bz) &\eqdef \bGamma \left[\begin{array}{c}
 \nabla_{\bx}  \cL_{p}(\bz) \\
 - \nabla_{\tiny \blambda}  \cL_{p}(\bz)
\end{array}
\right] \\
&= \bGamma \left[\begin{array}{c}
\bA \bx + \bb + \bD^\sfT \blambda + p \bGamma_x^{-1} \bx  \\
 - \bD \bx - \bd + p \bGamma_{\lambda}^{-1} \blambda
\end{array}
\right] \\
&= \bGamma \left (
\begin{bmatrix}
\bA & \bD^\sfT \\
-\bD & \mathbf{0}
\end{bmatrix}
\bz + 
p \bGamma^{-1} \bz +
\begin{bmatrix}
    \bb \\ -\bd
\end{bmatrix}
\right ) \\
&= (\bGamma \bW + p \bI) \bz + \bGamma \bw,
\end{align}
where we have defined
\begin{equation}
    \bW \eqdef \begin{bmatrix}
\bA & \bD^\sfT \\
-\bD & \mathbf{0}
\end{bmatrix}, \quad \bw \eqdef \begin{bmatrix}
    \bb \\ -\bd
\end{bmatrix}.
\end{equation}

\begin{proposition} \label{prop:sm-pd}
Let
\begin{equation}
    \bV \eqdef \frac{1}{2} \Big (\bGamma \bW  + \bW^\sfT \bGamma    \Big)
\end{equation}
and let $\lambda_{\min}$ be the smallest eigenvalue of $\bV$. If $p > \max \{0, -\lambda_{\min} \}$, then $\tilde{\Phi}$ is strongly monotone with parameter $\eta := p + \min \{0, \lambda_{\min} \}$.
\end{proposition}
The proof of Proposition \ref{prop:sm-pd} follows that of Proposition \ref{prop:sm-gd}.

\subsection{Convergence Analysis of Time-Varying OFO Algorithms} \label{sec:conv-ofo-hetero}

Given the above results, a similar tracking result to that of Theorem \ref{thm:conv} can be established. In particular, the tracking is now with respect to the modified reference trajectory $\{\bz^{(*, k)}\}$ which are the saddle points of the regularized Lagrangian \eqref{eq:regLagHet}. The analysis is based on the standard time-varying fixed point convergence analysis \cite{fixedPoints} and is omitted here for brevity.

\section{Application to Power Systems}
\label{sec:power_systems}

We next show two applications in the context of power systems, where heterogeneous  step sizes naturally arise. The target optimization problem \eqref{eqn:sampledProblem} is the linearized optimal power flow (OPF) problem as, e.g., in \cite{unified}.
The OFO framework \eqref{eq:dualstep} \eqref{eq:algo_distr} is set up in a hierarchical structure (see Figure \ref{fig:ofo_diagram}) where a coordinator collects the measurements of the network outputs $\hat{\by}^{(k)}$, keeps track of the output constraints \eqref{eq:ineqconst}, calculates their associated dual variables $\blambda^{(k)}$, and broadcasts them to the $N$ local controllers. In turn, each local controller $i$ computes power injection set points $\bx^{(k)}_i$ for distributed energy resource (DER) $i$ based on its cost function $f^{(k)}_i$, its feasible set $\cX_i^{(k)}\subset\mathbb{R}^2$, and the dual variables received from the coordinator.

\begin{figure}
    \centering
    \includegraphics[width=0.95\columnwidth]{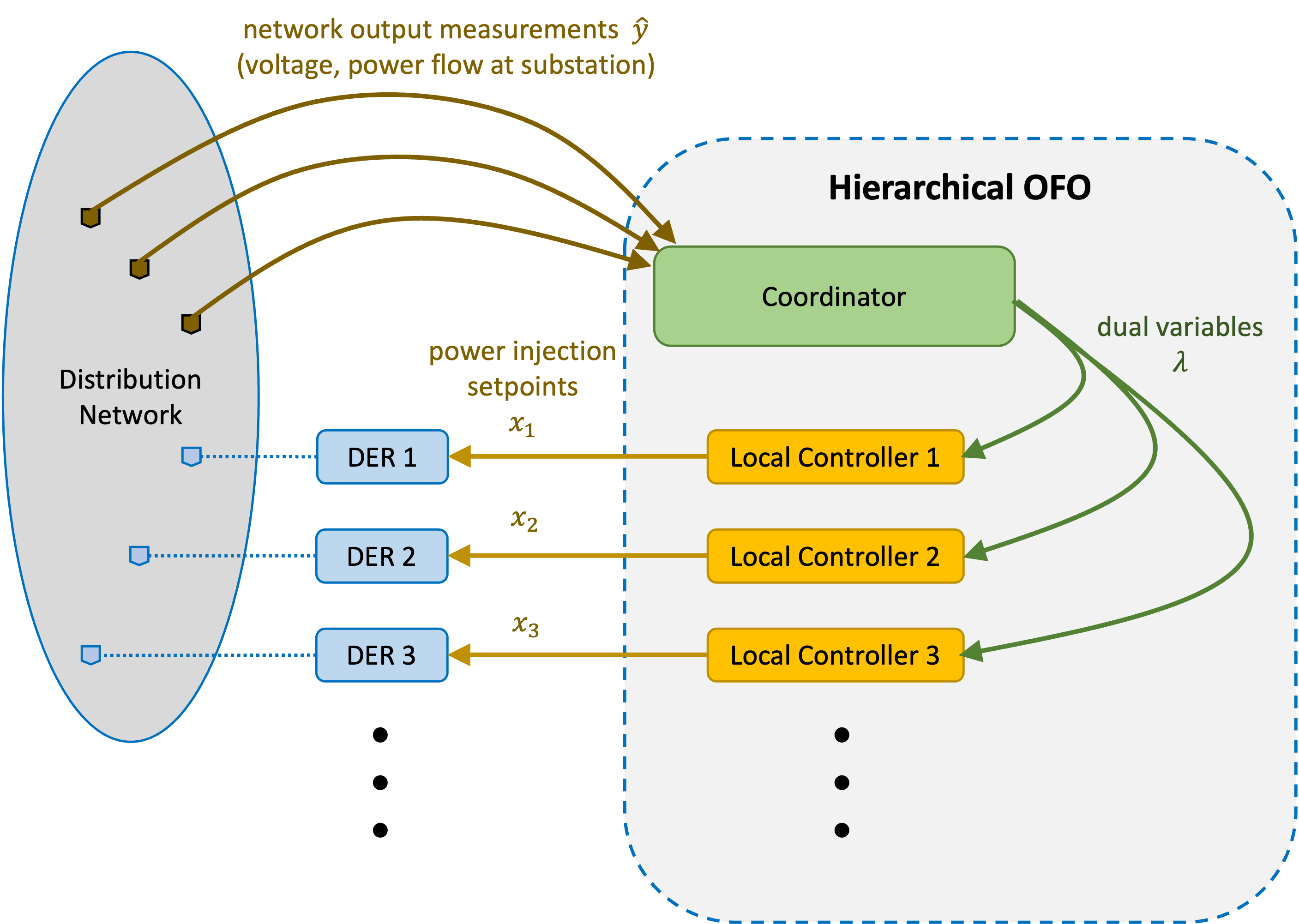}
    \caption{Hierarchical control architecture of DERs in a power distribution network using OFO.}
    \label{fig:ofo_diagram}
\end{figure}

There are two types of output constraints \eqref{eq:ineqconst} in the OPF problem associated with two different grid services:
\begin{itemize}
    \item \textbf{Virtual power plant (VPP).} Vector $\by$ contains the aggregate power flow at the distribution feeder head (substation), and a constraint is added to enforce tracking of a given aggregate power setpoint. 
    \item \textbf{Voltage regulation.} Vector $\by$ contains voltage magnitudes at different buses of the distribution network, and constraints are added to  keep voltage magnitudes between prescribed bounds.
\end{itemize}
An estimate of the Jacobian matrix $\bC$ in \eqref{eq:sys} is determined by linearizing the power flow equations as, e.g., in~\cite{bernstein2017linear}.
The cost functions $f_i^{(k)}$ of DERs are modeled with a quadratic function, quantifying the distance from a prescribed setpoint.
For example, for a solar panel, the cost can be measured as the squared amount of active power curtailed from the maximum available power being generated from the panel.
For an energy storage device or electric vehicle, it can be the squared difference between the injected active power and power to get to a desired state of charge.
The feasible set $\cX_i^{(k)}$  depends on the state of the DER.
More details of DER models can be found in~\cite{unified}.




\subsection{Adaptive Step Sizes}

One of the challenges to implement OFO on a real distribution system is that the step sizes need to be decided which heavily determine its performance.
If it were necessary to choose a common time-invariant step size, then a centralized operator would need to measure the performance of all the systems and constraints while finding a common step size that strikes a balance between them.
Allowing each system and constraint to have its own time-varying step size opens up the possibility that each system and constraint can automatically find their own step sizes based on local conditions. Specifically, \emph{adaptive} step sizes approach is proposed in \cite{yokota2017, comden2023adaptive}, as summarized next.

The adaptivity of the step sizes is based on two intuitive observations: (i) if the primal/dual variables are moving in the same direction for a period of time, then it likely indicates that they are heading towards a similar target point that is far away and speeding up the rate to get there would benefit the system; (ii) if the variables are oscillating, then it likely indicates that they are continually overshooting a target point and should dampen their action speeds.
Additionally, each system of primal variables and grouping of dual variables, based on similarities between associated constraints, should independently adapt their step sizes.

Let $\Omega$ be the set of primal and dual variables grouped for adapting step sizes, i.e., $\Omega:=\{\bx_1,\dots,\bx_N,\blambda_\text{volt},\blambda_\text{VPP}\}$ where $\blambda_\text{volt}$ are the dual variables associated with voltage regulation and $\blambda_\text{VPP}$ are the dual variables associated with VPP.
The gradient of the Lagrangian with respect to a specific set of variables $\omega\in\Omega$ gives the direction  that the algorithm is heading in, and thus comparing the gradients of the Lagrangian between consecutive iterations indicates a change or continuity in directions.
We use the cosine similarity metric to measure how much the directions change for $\omega\in\Omega$ at iteration $k$:
\begin{align}
    s_{\omega}^{(k)}:= \frac{\nabla_{\omega} \cL_p( \bx^{(k)}, \blambda^{(k)})^\sfT \nabla_{\omega} \cL_p( \bx^{(k-1)}, \blambda^{(k-1)})}{\|\nabla_{\omega} \cL_p( \bx^{(k)}, \blambda^{(k)})\| \|\nabla_{\omega} \cL_p( \bx^{(k-1)}, \blambda^{(k-1)})\|} \label{eq:adaptive_rule0}
\end{align}
where $s_{\omega}^{(k)}$ is between -1 and 1.
A value closer to 1 indicates that the directions are similar; whereas, a value closer to -1 indicates that the directions are opposite.

The cosine similarity at an iteration is used in a threshold rule to decide whether to proportionally increase, decrease or keep the step size the same:
\begin{align}
    \bgamma_\omega^{(k)} := \begin{cases}
        \overline{\kappa}_{\omega}\bgamma_\omega^{(k-1)} & \text{if }  s_{\omega}^{(k)}>\overline{s}_{\omega} \\ \underline{\kappa}_{\omega}\bgamma_\omega^{(k-1)} & \text{if }  s_{\omega}^{(k)}<\underline{s}_{\omega} \\
        \bgamma_\omega^{(k-1)} & \text{otherwise}
    \end{cases} \label{eq:adaptive_rule}
\end{align}
where $\overline{\kappa}_{\omega}>1$ is the increase factor, $\underline{\kappa}_{\omega}<1$ is the decrease factor, and $(\overline{s}_{\omega},\underline{s}_{\omega})$ are the thresholds.
The settings of the factors depend on the application, but typically it is good for $(\overline{\kappa}_{\omega}-1)<(1-\underline{\kappa}_{\omega})$ so that the speed up is gradual and the slow down to avoid oscillations is quick. The stepsizes are then used in \eqref{eq:mod_pd}, with $\bGamma_{x, i}^{(k)} := \diag(\bgamma_{\bx_i}^{(k)})$ and $\bGamma_{\lambda}^{(k)} := \diag(\blambda_\text{volt}^{(k)}, \blambda_\text{VPP}^{(k)})$.

As a demonstration, we apply OFO with adaptive step sizes to a three-phase power distribution system (see \cite{comden2023adaptive} for more details and results).
The following parameters for the adaptive step size rule are given the common settings for all sets of variables $\omega\in\Omega$: $\underline{s}_{\omega}:=0$, $\overline{s}_{\omega}:=0.9$, and $\overline{\kappa}_{\omega}:=1.005$.
The parameters for $\underline{\kappa}_{\omega}$ are differentiated for the dual variables: $\underline{\kappa}_{\blambda_\text{volt}}:=0.995$ and $\underline{\kappa}_{\blambda_\text{VPP}}:=0.5$ so that step size for VPP decreases much faster than that for voltage regulation when oscillations are observed.
This makes the voltage regulation constraints have a stronger adherence than the VPP constraints.
Finally, the DER step size parameter $\underline{\kappa}_{\bx_i}:\forall i\in\{1,\dots,N\}$ is set to 0.95.

\begin{figure}
     \centering
     \begin{subfigure}[b]{0.95\columnwidth}
         \centering
         \includegraphics[width=\columnwidth]{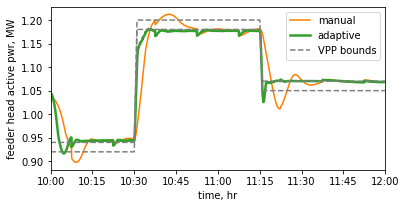}
         \caption{Phase A}
         \label{fig:vpptrack_PA}
     \end{subfigure}
     \hfill
     \begin{subfigure}[b]{0.95\columnwidth}
         \centering
         \includegraphics[width=\columnwidth]{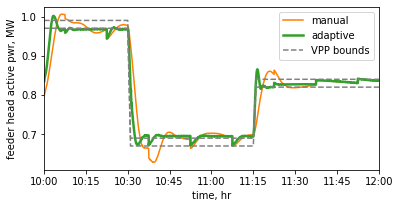}
         \caption{Phase B}
         \label{fig:vpptrack_PB}
     \end{subfigure}
     \hfill
     \begin{subfigure}[b]{0.95\columnwidth}
         \centering
         \includegraphics[width=\columnwidth]{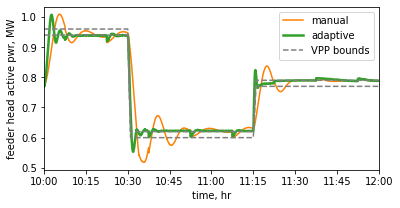}
         \caption{Phase C}
         \label{fig:vpptrack_PC}
     \end{subfigure}
        \caption{VPP set point tracking: manually tuned versus adaptively tuned step sizes.}
        \label{fig:vpptrack_P}
\end{figure}

\begin{figure}
     \centering
     \begin{subfigure}[b]{0.95\columnwidth}
         \centering
         \includegraphics[width=\columnwidth]{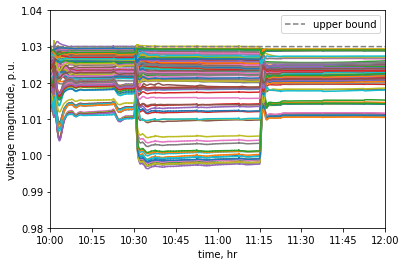}
         \caption{Adaptive step size tuning}
         \label{fig:vpptrack_auto_voltages}
     \end{subfigure}
     \hfill
     \begin{subfigure}[b]{0.95\columnwidth}
         \centering
         \includegraphics[width=\columnwidth]{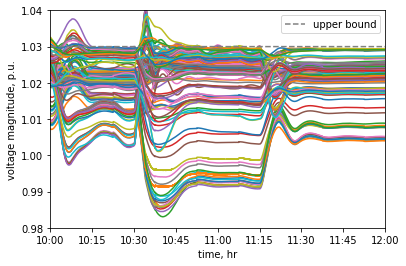}
         \caption{Manual step size tuning}
         \label{fig:vpptrack_man_voltages}
     \end{subfigure}
        \caption{Voltages magnitudes when the VPP set points have step changes.}
        \label{fig:vpptrack_voltages}
\end{figure}

\begin{figure}
     \centering
     \begin{subfigure}[b]{0.95\columnwidth}
         \centering
         \includegraphics[width=\columnwidth]{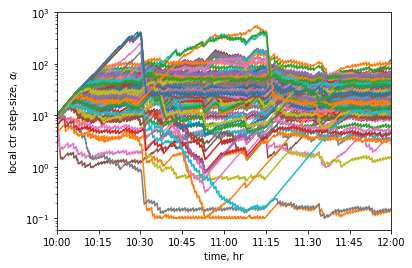}
         \caption{DERs}
         \label{fig:vpptrack_auto_stepsize_LCs}
     \end{subfigure}
     \hfill
     \begin{subfigure}[b]{0.95\columnwidth}
         \centering
         \includegraphics[width=\columnwidth]{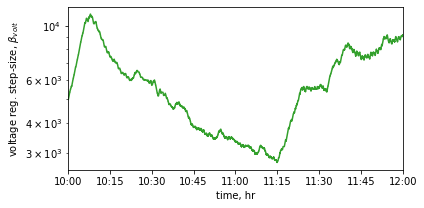}
         \caption{Voltage support}
         \label{fig:vpptrack_auto_stepsize_volt}
     \end{subfigure}
     \begin{subfigure}[b]{0.95\columnwidth}
         \centering
         \includegraphics[width=\columnwidth]{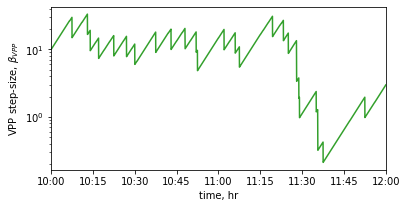}
         \caption{VPP}
         \label{fig:vpptrack_auto_stepsize_VPP}
     \end{subfigure}
        \caption{Step sizes of the DERs and the grid services during step changes in the VPP set points.}
        \label{fig:vpptrack_auto_stepsize}
\end{figure}

In the scenario we showcase here, the VPP setpoint makes two step changes and then the OFO controller needs to coordinate the DERs to change the feeder head power while keeping the voltage magnitudes within bounds.
We compare using the best manually tuned common step size with step sizes from the rule in  \eqref{eq:adaptive_rule0} and \eqref{eq:adaptive_rule}.
Figure \ref{fig:vpptrack_P} shows how well the feeder head power follows the VPP bounds.
The trajectories using the adaptive step sizes follows the bounds much faster and with less oscillations than those using the manually tuned step size.
Figure \ref{fig:vpptrack_voltages} shows how well the voltage magnitudes are being kept within the bounds.
The voltage magnitude trajectories using the adaptive step sizes has almost no violation and very little oscillations as compared those using the manually tuned step size.

Finally, we show in Figure \ref{fig:vpptrack_auto_stepsize} how the step sizes change over time for the DERs and the grid services.
It is easy to see that the heterogeneity of the step sizes and of the directions they go when the VPP bounds change are what enhances OFO to perform better than when using a manually tuned common time-invariant step size.

\subsection{Zero-Order OFO}
We next briefly mention the second application where heterogeneous step sizes naturally arise, which is the OFO with zero-order gradient estimation, as, e.g., in \eqref{eq:grad_proxy2} \eqref{eq:zeroOrderGradients}. To illustrate the idea, we  present the results from \cite{chen2020} regarding the approximate gradient \eqref{eq:grad_proxy2}.
Specifically,  under the assumptions in \cite{chen2020}, the zero-order projected gradient descent is approximated via the following \emph{averaged} iteration
\begin{align}\label{eq:zero-order-iter}
   \bx^{(k+1)}    &=\proj_{\cX^{(k)}} \Big \{\bx^{(k)} \notag \\
   &\qquad - \alpha     ( \Gamma^{(k)}\nabla F^{(k)}(\bx^{(k)}) +   O(  \alpha+  \epsy^2 ))  \Big \}
\end{align}
where
\begin{equation} \label{eq:exploration}
    \Gamma^{(k)} := \int_{kT }^{k(T+1)} \bfprobe(\tau) \bfprobe(\tau)^\sfT \, d\tau.
\end{equation}
In \eqref{eq:exploration}, $T$ is the averaging period, and $\bfprobe(t)$ is a continuous-time exploration signal. For example, if $\bfprobe(t)$ is chosen as a set of sinusoidal signals with distinct periods and amplitudes, and $T$ is chosen as a common integer multiplier of
the periods, the matrix $\Gamma^{(k)} $ will be diagonal; see \cite{chen2020} for details. Therefore, iteration \eqref{eq:zero-order-iter} becomes an approximate projected gradient algorithm with heterogeneous step sizes and the analysis in this paper applies.

\section{Conclusion}
\label{sec:conclusion}

In this paper, we presented an overview of the OFO framework and the recent advances on the topic. We also provided new theoretical insights on primal-dual OFO algorithms with heterogeneous step sizes for different elements of the gradient, and illustrated two applications in the power systems control context. 

Most of the new analysis is restricted to quadratic programs, and the extension of the approach to general  optimization problems remains an open research topic. The proposed regularization approach is effective in establishing strong monotonicity but leads to an analysis with respect to perturbed optimal saddle points; establishing tracking properties by bypassing the regularization (and perhaps strong monotonicity property altogether) is an open problem. Finally, an interesting direction is to consider the case where the matrix $\Gamma^{(k)}$ is a general positive definite, or even positive semi-definite matrix. The semi-definite case is of particular interest in the context of exact (rather than average) analysis of zero-order methods \cite{chen2020} and of asynchronous feedback gradient methods \cite{poolla2022}. 

\bibliographystyle{IEEEtran}
\bibliography{Reference,strings,markov,biblio,biblio1,biblio2}

\end{document}